\documentclass[12pt]{amsart}
\usepackage{amssymb}
\usepackage{enumerate}
\usepackage{amsmath,amssymb,amsfonts,amsthm, bm,
	latexsym, amscd, epsfig, epic, xcolor}
\usepackage{mathtools}
\usepackage{parskip} 
\usepackage{extarrows}
\usepackage{enumitem}
\usepackage{enumerate}
\usepackage{mathrsfs}
\usepackage{eucal}
\usepackage{xspace}
\usepackage{a4wide}
\usepackage{colordvi}
\usepackage{comment}
\usepackage{hyperref,eucal}
\usepackage{tikz}\usetikzlibrary{matrix,arrows}
\newcommand{\tcm}[1]{\textcolor{black}{#1}}
\newcommand{\emma}[1]{\textcolor{black}{#1}}

\newcommand{\emmaj}[1]{\textcolor{black}{#1}}

\theoremstyle{plain}
\newtheorem{theorem}{Theorem}[section]

\newtheorem*{main}{Main~Theorem}
\newtheorem{lemma}[theorem]{Lemma}

\theoremstyle{definition}
\newtheorem{definition}[theorem]{Definition}
\newtheorem{example}[theorem]{Example}

\def\D{\CMcal{D}}

\def\T{\CMcal{T}}

\def\A{\mathcal{A}}

\newcommand{\tdr}{\overline{\tau^{-1}\delta}}
\newcommand{\td}{\tau^{-1}\delta}
\newcommand{\id}{\mathrm{id}}
\newcommand{\suchthat}{\,:\,}
\newcommand{\spam}{\operatorname{span}}
\newcommand{\bdet}{\bm{\mathrm{det}}}
\newcommand{\bJT}{\bm{{JT}}}

\newcommand{\Sym}{\ensuremath{\operatorname{Sym}}}

\newcommand{\NCSym}{\ensuremath{\operatorname{NCSym}}}
\newcommand{\slashp}{\mid}

\newlength\cellsize 
\savebox2{%
	\begin{picture}(15,15)
		\put(0,0){\line(1,0){15}}
		\put(0,0){\line(0,1){15}}
		\put(15,0){\line(0,1){15}}
		\put(0,15){\line(1,0){15}}
\end{picture}}
\newcommand\cellify[1]{\def\thearg{#1}\def\nothing{}%
	\ifx\thearg\nothing
	\vrule width0pt height\cellsize depth0pt\else
	\hbox to 0pt{\usebox2\hss}\fi%
	\vbox to 15\unitlength{
		\vss
		\hbox to 15\unitlength{\hss$#1$\hss}
		\vss}}
\newcommand\tableau[1]{\vtop{\let\\=\cr
		\setlength\baselineskip{-16000pt}
		\setlength\lineskiplimit{16000pt}
		\setlength\lineskip{0pt}
		\halign{&\cellify{##}\cr#1\crcr}}}
\savebox3{%
	\begin{picture}(15,15)
		\put(0,0){\line(1,0){15}}
		\put(0,0){\line(0,1){15}}
		\put(15,0){\line(0,1){15}}
		\put(0,15){\line(1,0){15}}
\end{picture}}
\newcommand\expath[1]{%
	\hbox to 0pt{\usebox3\hss}%
	\vbox to 15\unitlength{
		\vss
		\hbox to 15\unitlength{\hss$#1$\hss}
		\vss}}
\newcommand\bas[1]{\omit \vbox to \cellsize{ \vss \hbox to \cellsize{\hss$#1$\hss} \vss}}

\theoremstyle{remark}

\numberwithin{equation}{section}
\numberwithin{figure}{section}

\title[Equal skew Schur functions in noncommuting variables]{Equality of skew Schur functions in noncommuting variables}

\author{Emma Yu Jin and Stephanie van Willigenburg}
\address{School of Mathematical Sciences, Xiamen University, Xiamen, Fujian 361005, China}
\email{yjin@xmu.edu.cn}
\address{Department of Mathematics, University of British Columbia, Vancouver, British Columbia V6T 1Z2, Canada}
\email{steph@math.ubc.ca}


\date{\today}

\begin{document}
	
	\begin{abstract}
		The question of classifying when two skew Schur functions are equal is a substantial open problem, which remains unsolved for over a century. In 2022, Aliniaeifard, Li and van Willigenburg introduced skew Schur functions  in noncommuting variables, $s_{(\delta,\D)}$, where $\D$ is a connected skew diagram with $n$ boxes and $\delta$ is a permutation in the symmetric group $S_n$.  
		
		In this paper, we combine these two and classify  when two skew Schur functions in noncommuting variables are equal:  $s_{(\delta,\D)} = s_{(\tau,\T)}$ such that $\D\ne \T$  if and only if $\D$ is a nonsymmetric ribbon, $\T$ is the antipodal rotation of $\D$ and
		$\tdr$ is an explicit bijection between two set partitions determined by $\D$.
	\end{abstract}
	\maketitle
	
	\section{Introduction}\label{S:intro}
	
	The problem of classifying when two skew Schur functions are equal has been open since Schur introduced them in 1901 \cite{s:01}. While partial progress has been made \cite{ mw:14, mw:09,  rsw:07, y:17}, only a few special cases have been classified, for example \cite{btw:06, gut:09, ly:24, w:05}. The most notable of these classifications was that of ribbon Schur functions \cite{btw:06}, where the classification had further impact \cite{aww:21, acsz:20, az:14,   bdf:10}. Another long-standing problem was to find a basis analogous to that of Schur functions for the algebra of symmetric functions in noncommuting variables. This problem was posed in 2004 by Rosas and Sagan \cite{rs:04}, and was eventually resolved in 2022 by Aliniaeifard, Li and van Willigenburg \cite{alw:22}. In this short paper we combine these two problems and classify succinctly when two skew Schur functions in noncommuting variables are equal:
	
	\begin{main} Given two connected skew diagrams $\D$ and $\T$ such that $\D\ne \T$, we have that the skew Schur functions in noncommuting variables $$s_{(\delta, \D)} =s_{(\tcm{\tau},\T)}$$
		if and only if 
		\begin{enumerate}
			\item $\D$ is a nonsymmetric ribbon and
			\item $\T=\D^*$, and
			\item the bijection $\tdr : j\mapsto n+1-\td(j)$ preserves each block of the set partition $[\alpha]$ where $n=\vert \D\vert$ and $\alpha$ is the row-length composition of   $\D$.
		\end{enumerate}
	\end{main} 
	\emma{ We remark that the characterization of $s_{(\delta,\D)}=s_{(\tau,\T)}$ for skew diagrams $\D$ and $\T$ is reduced to studying this equality for connected skew diagrams $\D$ and $\T$, which is to be explained at the end of Section \ref{S:background}.}
	
	In particular, in Section~\ref{S:background}, we give all the background needed to understand the above theorem, and prove some small but valuable lemmas. Then we devote Section~\ref{S:proof} to proving the above theorem, which we restate just before the proof as Theorem~\ref{T:classification}.
	
	\section{Background}\label{S:background}
	
	We begin by reviewing the combinatorial concepts we need, before introducing our algebras and functions of study.
	
	\subsection{Compositions, partitions and set partitions}\label{ss:21}
	A \emph{composition} $\alpha = (\alpha_1, \ldots, \alpha_l)$ of $n$ is a finite sequence of positive integers $\alpha_1, \ldots, \alpha_l$ such that $\sum _{i=1}^l \alpha_i = n$, denoted by $\alpha \vDash n$. We call the $\alpha_i$ the \emph{parts} of $\alpha$, call $\ell(\alpha)=l$ the \emph{length} of $\alpha$ and call $|\alpha|=n$ the \emph{size} of $\alpha$. We denote by $0$ the unique composition of length and size 0. If the parts of $\alpha$ appear in weakly decreasing order, then we call this a \emph{partition} $\lambda$, denoted by $\lambda \vdash n$, and if the parts of $\alpha$ are allowed to include 0 then we call this a \emph{weak} composition. Note that every composition $\alpha$ determines a partition $\lambda(\alpha)= (\lambda(\alpha)_1, \ldots , \lambda(\alpha)_{\ell(\alpha)})$ obtained by listing the parts of $\alpha$ in weakly decreasing order. Given a composition $\alpha = (\alpha_1, \ldots, \alpha_{\ell(\alpha)})$ we define
	$$\alpha!=\alpha_1!\cdots \alpha_{\ell(\alpha)}! \textrm{ and } \alpha^*=(\alpha_{\ell(\alpha)},\ldots ,\alpha_{1}).$$
	Meanwhile, given two compositions of \tcm{$n$,} \emma{ say }$\alpha = (\alpha_1, \ldots, \alpha_{\ell(\alpha)}) $ and $\beta= (\beta_1, \ldots, \beta_{\ell(\beta)})$, we state that $\beta$ is a \emph{coarsening} of $\alpha$ (or $\alpha$ is a \emph{refinement} of $\beta$), denoted by $\beta\succcurlyeq \alpha$, if $\beta$ is obtained from $\alpha$ by adding together adjacent parts of $\alpha$, and $\beta$ \emph{dominates} $\alpha$ if $\sum _{j=1}^i \beta_j \geq \sum _{j=1}^i \alpha_j$ for all $1\leq i \leq \min\{\ell(\alpha), \ell(\beta)\}$.
	
	\begin{example}\label{EX:compositions}
		If $\alpha = (1,2,1,3,2) \vDash 9$ then $\ell(\alpha) = 5$, $|\alpha| = 9$ and $\lambda(\alpha) = (3,2,2,1,1)$. Note that $\alpha ! = 1!2!1!3!2! = 24$, $\alpha ^\ast = (2,3,1,2,1)$, and $(1,3,3,2) \succcurlyeq (1,2,1,3,2)$.
	\end{example}
	
	Given a partition $\lambda = (\lambda_1, \ldots, \lambda_{\ell(\lambda)})$, we say that its \emph{diagram}, also denoted by $\lambda$, is the array of left-justified boxes with $\lambda_i$ boxes in row $i$ from the top. Given two partitions $\lambda =   (\lambda_1, \ldots, \lambda_{\ell(\lambda)}) \vdash n$ and $\mu = (\mu_1, \ldots, \mu_{\ell(\mu)}) \vdash m$ such that $\ell (\mu) \leq \ell(\lambda)$ and $\mu _i \leq \lambda _i$ for all $1 \leq i \leq \ell(\mu)$ we say that $\mu$ is \emph{contained} in $\lambda$, denoted by $\mu \subseteq \lambda$, and moreover the \emph{skew diagram} $\lambda / \mu$ of \emph{size} $(n-m) =| \lambda / \mu |$ is the array of boxes contained in $\lambda$ but not in $\mu$ when the array of boxes of $\mu$ is positioned in the top-left corner of the array of boxes of $\lambda$.  Given a skew diagram $\lambda / \mu$ we define $(\lambda / \mu)^\ast$ to be the $180^\circ$ antipodal rotation of $\lambda / \mu$, and we say that it is \emph{symmetric} if $\lambda / \mu = (\lambda / \mu)^\ast$. 
	
	We say that a skew diagram is \emph{connected} if each pair of adjacent rows overlap in at least one column, and say that it is a \emph{ribbon} if each pair of adjacent rows overlap in \emph{exactly} one column. \emma{ See Example \ref{EX:rowoverlap} below.}
	\begin{example}\label{EX:rowoverlap}
		The skew diagram for \emma{ $\D = (5,5,5,4,4,2)/(4,4,3,3,1)$} is a ribbon where $|\D|=10$ and is the skew diagram on the left, while $\D^\ast$ is the skew diagram on the right.
		$$\D = \tableau{&&&&\ \\&&&&\ \\&&&\ &\ \\&&&\ \\&\ &\ &\ \\\ &\ } \qquad \D^\ast = \tableau{&&&\ &\ \\ &\ &\ &\ \\&\ \\\ &\ \\\ \\\ }$$
	\end{example}

	How much \tcm{certain sets of} rows overlap will be important for the rest of our paper and so we define \tcm{this} formally now, in addition to often referring to skew diagrams by letters such as $\D$ in order to streamline notation.
	
	\begin{definition}\label{D:rowoverlap}
		Let $\D$ be a skew diagram occupying $r$ rows. For each $k\in \tcm{\{1,\ldots,r\}}$, we define the \emph{$k$-row overlap composition}  to be the (weak) composition
		\begin{align*}
			\alpha^{(k)}(\D) =(r_1^{(k)},\ldots,r_{r-k+1}^{(k)})
		\end{align*}
		where $r_i^{(k)}$ is the number of columns occupied in common by the rows $i,i+1,\ldots,i+k-1$. Let $\lambda^{(k)}(\D)$ be the \emph{$k$-row overlap partition}, that is, the partition determined by rearranging the parts of $\alpha^{(k)}(\D)$ in weakly decreasing order. Furthermore, the \emph{row-length composition} is $\alpha(\D) = \alpha^{(1)}(\D)$, and the \emph{row-length partition} is $\lambda(\D) = \lambda(\alpha(\D))$.
	\end{definition}
	\begin{example}
		\emma{ For the skew diagram $\D$ in Example \ref{EX:rowoverlap}}, 
		note that $\alpha(\D) = \alpha^{(1)}(\D) = (1,1,2,1,3,2)$, $\alpha^{(2)}(\D) = (1,1,1,1,1)$, $\alpha^{(3)}(\D) = \emmaj{(1,0,1,0)}, \alpha^{(4)}(\D) = (0,0,0), \alpha^{(5)}(\D) = (0,0), \alpha^{(6)}(\D) = (0)$ and $\lambda(\D) = (3,2,2,1,1,1)$.
	\end{example}
	
	Observe that for a skew diagram \tcm{$\D$,}
	$$\alpha (\D ^\ast) = (\alpha (\D))^\ast$$and that if $\D$ is a ribbon, then because we know that every pair of adjacent rows overlap in exactly one column we have that $\alpha(\D)$ determines the ribbon exactly, and hence there is a natural bijection between ribbons of size $n$ and compositions of size $n$.
	
	We now move from diagrams to set partitions. Given $[n] = \{1, \ldots, n\}$, a \emph{set partition} $\pi = \pi_1 / \cdots / \pi_l$ of $[n]$ is a family of disjoint nonempty subsets of positive integers $\pi_1, \ldots, \pi_l$ such that $\cup _{i=1}^{l} \pi_i = [n]$, denoted by $\pi \vdash [n]$. We call the $\pi_i$ the \emph{blocks} of $\pi$, call $\ell(\pi)=l$ the \emph{length} of $\pi$, and call $|\pi|=n$ the \emph{size} of $\pi$. We usually list the blocks by increasing least element, omitting set parentheses and commas for ease of legibility, and denote by $\emptyset$ the unique set partition of length and size 0. Note that every set partition $\pi$ determines a partition $\lambda (\pi) = (\lambda(\pi) _1, \ldots , \lambda (\pi)_{\ell(\pi)})$, obtained by listing the cardinalities of the block sizes of $\pi$ in weakly decreasing order, and we set $\lambda (\pi)!=\lambda (\pi)_1!\cdots \lambda (\pi)_{\ell(\pi)}!$.
	
	\begin{example}\label{EX:setpartitions}
		Note that the family of disjoint subsets $\{1\}, \{2,4\}, \{3\}, \{5,6,7\}, \{8,9\}$ is a set partition of $[9]$, and as a set partition $\pi$ we write $$\pi = 1/24/3/567/89 \vdash [9]$$with $\ell(\pi)=5, |\pi| = 9$, $\lambda (\pi) = (3,2,2,1,1)$ and $\lambda (\pi)! = 3!2!2!1!1!=24$.
	\end{example}
	
	Given two set partitions $\pi \vdash [n]$ and $\sigma = \sigma _1 / \cdots / \sigma _{\ell(\sigma)} \vdash [m]$, we say that their \emph{slash product} $\pi \slashp \sigma$ is
	\begin{align}\label{E:slashprod}
		\pi\slashp \sigma=\pi/(\sigma_1+n)/\cdots/(\sigma_{\ell(\sigma)}+n)\vdash [n+m],
	\end{align} where $\sigma _i+n = \{ s+n \suchthat s\in \sigma _i\}$ for $1\leq i \leq \ell(\sigma)$. \tcm{We also say that two set partitions $\pi, \sigma \vdash [n]$ satisfy $\pi\le \sigma$ if $\sigma$ is obtained from $\pi$ by merging blocks of $\pi$.} Returning to compositions, given a composition $\alpha = (\alpha_1, \ldots, \alpha_{\ell(\alpha)})\vDash n$, its corresponding set partition $[\alpha]\vdash [n]$ is
   \emma{
	\begin{align}\label{E:setalp}
		[\alpha]
		&=[\alpha_1]\,\vert\, [\alpha_2]\,\vert \cdots\vert\, [\alpha_{\ell(\alpha)}],
	\end{align}
    where $[\alpha_i]=\{1,\ldots,\alpha_i\}$.}
	We use $[\alpha_i]_{\alpha}$ to denote each block of $[\alpha]$, in order to avoid confusion \emma{ with the set $[\alpha_i]$.}
	
	\begin{example}\label{EX:slashp}
		If $\pi = 1/24/3\vdash [4]$ and $\sigma = 123/45 \vdash [5]$ then $\pi \slashp \sigma = 1/24/3/567/89 \vdash [9]$. Meanwhile if $\alpha = (1,2,1,3,2)$ then 
		$$[\alpha]= [1]\slashp [2] \slashp [1] \slashp [3] \slashp [2]= 1/23/4/567/89 $$with $[\alpha_4]_\alpha =\tcm{\{5,6,7\}}$.
	\end{example}
	
	We now turn our attention to the Hopf algebras of symmetric functions $\Sym$ and symmetric functions in noncommuting variables $\NCSym$, respectively. 
	
	\subsection{\emma{Symmetric functions in commuting variables}}
	The \emph{Hopf algebra of symmetric functions} $\Sym$ is the graded Hopf algebra
	$$\Sym = \Sym ^0 \oplus \Sym ^1 \oplus \cdots \subset \mathbb{Q} [[x_1, x_2, \ldots ]]$$where {$[[\cdot ]]$ means that the variables commute,} $\Sym ^0 = \spam \{1\}$ and the $n$th graded piece for $n\geq 1$ has the following  bases
	$$\begin{array}{rclcl}
		\Sym ^n &=&\spam\{  h_\lambda \suchthat \lambda\vdash n\}&=& \spam\{ s_\lambda \suchthat \lambda\vdash n\}
	\end{array}$$where these functions are defined as follows, given a partition $\lambda = (\lambda_1,\ldots, \lambda_{\ell(\lambda)})\vdash n$.
	
	The \emph{complete homogeneous symmetric function}, $h_\lambda$, is given by
	$$h_\lambda = h_{\lambda _1}\cdots h_{\lambda _{\ell(\lambda)}}$$where {$h_{i} = \sum_{\tcm{j_1\leq\cdots \leq j_{i}}} \tcm{x_{j_1}\cdots x_{j_{i}}}$.}
	
	\begin{example}\label{EX:h}
		$h_{(2,1)} = (x_1x_2+x_1^2+ \cdots)(x_1+x_2+\cdots)$
	\end{example}
	
	For the next basis we begin with two partitions $\lambda = (\lambda_1,\ldots, \lambda_{\ell(\lambda)})$ and $\mu = (\mu_1,\ldots, \mu_{\ell(\mu)})$ such that $\mu \subseteq \lambda$, and define the \emph{Jacobi-Trudi matrix} of the skew diagram $\lambda /\mu$ to be
	\begin{equation}\label{E:JTskew}
		JT(\lambda/\mu) =  \left( h_{\lambda _i  -\mu_j - i + j} \right) _{1\leq i,j \leq \ell(\lambda)} \end{equation}where we set $\mu_j = 0$ for $\ell(\mu) <j \leq \ell(\lambda)$,  $h_0=1$ and \emma{ $h_i=0$ for $i<0$}. Then the \emph{skew Schur function}, $s_{\lambda/\mu}$, is given by
	\begin{equation}\label{E:SSF}s_{\lambda/\mu} = \det	JT(\lambda /\mu)\end{equation}and if $\mu=0$ then the \emph{Schur function}, $s_\lambda$, is given by
	\begin{equation}\label{E:SF}s_{\lambda} = \det	JT(\lambda ).\end{equation}
	
	\begin{example}\label{EX:schurs}
		$s_{(2,1)} = \det \begin{pmatrix}h_2&h_3\\h_0&h_1\end{pmatrix} = h_2h_1-h_3h_0 = h_{(2,1)}-h_{(3)}$
	\end{example}	
	
	In particular, if $\lambda/\mu$ is a ribbon corresponding to the composition $\alpha$, then the \emph {ribbon Schur function}, \tcm{$r_\alpha$,} is given by \cite[Proposition 2.1]{btw:06}
	$$r_\alpha = (-1)^{\ell(\alpha)} \sum _{\beta \succcurlyeq \alpha} (-1) ^{\ell(\beta)} h_\beta.$$

	\begin{example}\label{EX:ribbons}
		$r_{(2,1)}  = h_{(2,1)}-h_{(3)}$
	\end{example}
	
	The Jacobi-Trudi matrix also satisfies some particularly useful properties.
	
	\begin{lemma}\label{L:1}
		Given a skew diagram $\lambda/\mu$, let $\mathcal{A}_{ij}=\lambda_i-\mu_j-i+j$ for any $i,j$. Then for any \tcm{$1\leq i<j\leq \ell(\lambda)$, and $1\leq k<m\leq \ell(\lambda)$,} we have \tcm{that} $\mathcal{A}_{ik}+\mathcal{A}_{jm}=\mathcal{A}_{im}+\mathcal{A}_{jk}$.
	\end{lemma}
	\begin{proof}
		This follows by the fact that both sides equal $\lambda_i+\lambda_j-\mu_k-\mu_m-i+k-j+m$.
	\end{proof}  
	
	\begin{lemma}\cite[Proposition 6.2]{rsw:07}
		\label{L:JTsub}
		Let $\lambda/\mu$ be a skew diagram with $\ell(\mu)<\ell=\ell(\lambda)$.
		\begin{enumerate}
			\item The largest subscript  occurring on any nonzero entry $h_{L}$ in the Jacobi-Trudi matrix $JT(\lambda/\mu)$ is $L=\lambda_1+\ell-1$
			and this subscript occurs exactly once, on the $(1,\ell)$-entry $h_{L}$.
			\item The subscripts on the diagonal entries in $JT(\lambda/\mu)$ are exactly  
			\begin{align*}
				\alpha(\lambda/\mu)=(\lambda_1-\mu_1,\ldots,\lambda_{\ell}-\mu_{\ell})
			\end{align*}
			and the monomial $h_{\lambda_1-\mu_1}\cdots h_{\lambda_{\ell}-\mu_{\ell}}$ occurs in the determinant $s_{\lambda/\mu}$
			\begin{enumerate}
				\item with coefficient $+1$, and
				\item as the monomial whose subscripts rearranged into weakly decreasing order give the smallest partition of $\vert\lambda/\mu\vert$ in dominance order among all nonzero monomials.
			\end{enumerate}    
		\end{enumerate}       
	\end{lemma}
	
	\subsection{\emma{Symmetric functions in noncommuting variables}}
	Meanwhile, the \emph{Hopf algebra of symmetric functions in noncommuting variables} $\NCSym$ is the graded Hopf algebra
	$$\NCSym  = \NCSym ^0 \oplus \NCSym ^1 \oplus \cdots \subset \mathbb{Q} \langle \langle x_1, x_2, \ldots \rangle\rangle$$where $\langle \langle\cdot \rangle\rangle$ means that the variables do not commute, $\NCSym ^0 = \spam \{1\}$ and {the} $n$th graded piece for $n\geq 1$ has the following bases
	$$\begin{array}{rclcl}
		\NCSym ^n&=& \spam\{  {h_\pi} \suchthat \pi\vdash [n]\}&=& \spam\{ s_\pi \suchthat \pi\vdash [n]\}
	\end{array}$$where these functions are defined  following \cite{alw:22}, given a set partition $\pi  = \pi _1/\cdots /\pi _{\ell(\pi)}\vdash [n]$.
	
	The \emph{complete homogeneous symmetric function in $\NCSym$}, $h_\pi$, is given by \cite[Lemma 2.14]{alw:22}
	$$
	h_{\pi}=\sum_{\varepsilon} \sum_{(i_1,\ldots,i_n)  } x_{i_{\varepsilon(1)}}\cdots x_{i_{\varepsilon(n)}}
	$$ 
	where
	\begin{enumerate}
		\item   the first sum is over all $\varepsilon\in S_n$ that fixes the blocks of $\pi$,
		\item   the second sum is over all $n$-tuples of positive integers $(i_1,\ldots,i_n)$ such that if $j$ and $k$ are in the same block of $\pi$ with $j<k$, then $i_j\leq i_k$.
	\end{enumerate} 
	
	\begin{example}\label{EX:hpi}
		$h_{13/2}= {2x_1x_1x_1+ x_1x_1x_2+x_2x_1x_1+2x_1x_2x_1}\tcm{+ x_1x_2x_3}+\cdots$
	\end{example}
	
	These functions multiply together in a natural way.
	
	\begin{lemma}\label{L:hmul}\cite[Corollary 2.41]{ber:10}\
		For set partitions $\pi$ and $\sigma$ we have that
		$$h_\pi h_\sigma = h_{\pi \slashp \sigma}.$$
	\end{lemma}
	
	Hence, much like their counterparts in $\Sym$, given a composition $\alpha =  (\alpha_1, \ldots, \alpha_{\ell(\alpha)}) $ we have that
	$$h_{[\alpha]}=h_{[\alpha_1]}\cdots h_{[\alpha_{\ell(\alpha)}]}.$$We also have that \cite[Section 2]{rs:04} given a permutation $\delta\in S_n$, \emma{ define}
	$$\delta \circ h_\pi = h _{\delta\pi}$$where $\delta$ acts on $\pi$ by sending every element $i$ to $\delta(i)$.
	
	\begin{example}\label{E:deltapi}
		$321\circ h_{12/3} = h_{1/23}$, where the permutation $321\in S_3$ is written in one-line notation.
	\end{example}
	
	For the next basis, much like classical Schur functions, we will again need to compute a determinant, however now it will need to be noncommutative, which we recall as follows \cite[Equation 2.6]{alw:22}. We define the \emph{noncommutative analogue of Leibniz' determinantal formula} for any matrix $A=(a_{ij}) _{1\leq i,j\leq n}$ with noncommuting entries $a_{ij}$ to be
	\begin{equation}\label{E:leibniz}
		\bdet (A) = \sum _{\varepsilon \in S _n} \mathrm{sgn} (\varepsilon) a_{1\varepsilon (1)}\cdots a_{n\varepsilon (n)}
	\end{equation}
	that takes the product of the entries from the top row to the bottom row, and $\mathrm{sgn} (\varepsilon)$ is the sign of the permutation $\varepsilon$. Given two partitions $\lambda = (\lambda_1,\ldots, \lambda_{\ell(\lambda)})$ and $\mu = (\mu_1,\ldots, \mu_{\ell(\mu)})$ such that $\mu \subseteq \lambda$, we define the \emph{noncommutative Jacobi-Trudi matrix} of the skew diagram $\lambda /\mu$ to be \begin{equation}\label{E:noncommJTskew}
		\bJT(\lambda/\mu) =\left(\frac{1}{(\lambda_i-\mu_j-i+j)!}h_{[\lambda_i-\mu_j-i+j]}\right)_{1\leq i,j\leq \ell(\lambda)}\end{equation}where we set $\mu_j = 0$ for $\ell(\mu) <j \leq \ell(\lambda)$,  \emma{ $h_{[0]}= h_\emptyset = 1$ and  $h_{[i]}=0$ for $i<0$}. If $|\lambda/\mu|= n $ and $\delta \in S_n$, then the \emph{skew Schur function in $\NCSym$}, $s_{(\delta, \lambda/\mu)}$, is given by \tcm{\cite[Definition 4.2]{alw:22}}
	\begin{equation}\label{E:noncommSSF}s_{(\delta, \lambda/\mu)} = \delta \circ \bdet\	\bJT(\lambda /\mu).\end{equation}If $\delta = \id$, then we call this the \emph{source skew Schur function in $\NCSym$},  $s_{[\lambda/\mu]}$, hence given by \cite[Definition 3.1]{alw:22}
	\begin{equation}\label{E:noncommsourceSSF}s_{[\lambda/\mu]} =  \bdet\	\bJT(\lambda /\mu) = \bdet \left(\frac{1}{(\lambda_i-\mu_j-i+j)!}h_{[\lambda_i-\mu_j-i+j]}\right)_{1\leq i,j\leq \ell(\lambda)}.\end{equation}
	
	\begin{example}\label{EX:spi}
		The source  Schur function in noncommuting variables $s_{[(2,1)]}$ is
		\begin{align*}s_{[(2,1)]} &  = {\bdet \begin{pmatrix} \frac{1}{2!} h_{[2]}& \frac{1}{3!} h_{[3]}\\
					\frac{1}{0!} h_{[0]}& \frac{1}{1!} h_{[1]}
			\end{pmatrix}} = \bdet \begin{pmatrix} \frac{1}{2!} h_{12}& \frac{1}{3!} h_{123}\\
				\frac{1}{0!} h_\emptyset & \frac{1}{1!} h_{1}
			\end{pmatrix}\\
			&= \frac{1}{2!} h_{12} \frac{1}{1!} h_{1} - \frac{1}{3!} h_{123}\frac{1}{0!} h_\emptyset = {\frac{1}{2} h_{12 \slashp 1} - \frac{1}{6} h_{123}} = \frac{1}{2} h_{12/3} - \frac{1}{6} h_{123}.
		\end{align*}
	\end{example}
	
	Source skew Schur functions have the following property.
	
	\begin{lemma}\label{L:2}
		For any connected skew diagram $\D$ of size $n$,  the source skew Schur function $s_{[\D]}$ is a linear combination of $h_{[\gamma]}$ with $\gamma \vDash n$. 
	\end{lemma}
	
	\begin{proof}
		Let $\D=\lambda/\mu$ with $|\lambda/\mu|=n$ \emma{ and $\ell=\ell(\lambda)$. Recall that $\mathcal{A}_{ij}=\lambda_i-\mu_{j}-i+j$. For any $w\in S_{\ell}$, define $w_{\A}=(\A_{1\,w(1)},\ldots, \A_{\ell\,w(\ell)})$ as a sequence of integers.}
		Then the determinantal formula of $s_{[\lambda/\mu]}$ gives that
		\emma{
			\begin{align}
				\nonumber s_{[\lambda/\mu]}&=\bdet\left((\lambda_i-\mu_j-i+j)!^{-1} h_{[\lambda_i-\mu_j-i+j]}\right)_{1\leq i,j \leq \ell}\\
				\label{E:skewsource}&=\sum_{w\in S_{\ell}}\mathrm{sgn}(w)\left(\prod_{i=1}^{\ell}(\A_{i\,w(i)}!)^{-1}\right)
				h_{[w_{\A}]}.
		\end{align}}
		\emma{ Only a nonzero $h_{[w_{\A}]}$ contributes to the sum and note that $h_{[w_{\A}]}\ne 0$ if and only if $\A_{i\,w(i)}\ge 0$ for all $i$, i.e., $w_{\A}$ is  a weak composition. Removing zeros from $w_{\A}$ produces a composition, say $\gamma$. Now by construction $\gamma\vDash \A_{1\,w(1)}+\cdots+ \A_{\ell\,w(\ell)}$ and
			$$\sum _{i=1}^{\ell} \A_{i\,w(i)} =\sum _{i=1}^{\ell} (\lambda_i-\mu_{w(i)}-i+w(i))= \sum _{i=1}^{\ell} (\lambda_i-\mu_i) =n$$}
		as desired, because $w$ is a permutation.
	\end{proof}
	
	Given $\pi \vdash [n]$ let us arrange the blocks such that reading from left to right
	\begin{enumerate}
		\item  block sizes weakly decrease, 
		\item smallest elements of blocks of the same size are strictly increasing, \emmaj{and}
		\item \emmaj{elements of each block are strictly increasing.}
	\end{enumerate} 
	Let $\delta _\pi$ be the permutation in one-line notation obtained by removing the slashes from $\pi$ with the blocks arranged as just described. \emma{ An important point to note is that the above ordering of blocks is different from the one by increasing least element of each block in Subsection \ref{ss:21}. We distinguish them by indexing permutations exclusively via the above ordering, that is an arrangement of blocks according to \emmaj{(1)--(3)}.}

	Then the \emph{(standard) Schur function in $\NCSym$}, $s_\pi$, is given by \cite[Definition 4.2]{alw:22}
	\begin{equation}\label{E:noncommSF}s_{\pi} = \delta_\pi \circ s_{[\lambda(\pi)]} = \delta _\pi \circ \bdet \left(\frac{1}{(\lambda(\pi)_i-i+j)!}h_{[\lambda(\pi)_i-i+j]}\right)_{1\leq i,j\leq \ell(\pi)}.\end{equation}
	
	\begin{example}
		\emma{ For the source Schur function $s_{[(2,1)]}$ in Example \ref{EX:spi}, \emmaj{(\ref{E:noncommSSF}) gives that} $$s_{(213,(2,1))}=213\circ s_{[(2,1)]}=213 \circ \left( \frac{1}{2} h_{12/3} - \frac{1}{6} h_{123}\right)= \frac{1}{2} h_{12/3} - \frac{1}{6} h_{123}.$$} 
		If $\pi = 13/2$, then $\delta _\pi = 132 \in S_3$ in one-line notation and so
		\emmaj{(\ref{E:noncommSF}) says that} $$s_{13/2} = 132 \circ s_{[(2,1)]} = 132 \circ \left( \frac{1}{2} h_{12/3} - \frac{1}{6} h_{123}\right) = \frac{1}{2} h_{13/2} - \frac{1}{6} h_{123}.$$
		Meanwhile the  skew Schur function in noncommuting variables $s_{(321, \tcm{(2,2)/(1)})}$, where $321\in S_3$ is in one-line notation, is
		\begin{align*}s_{(321, (2,2)/(1))} &=  321 \circ {\bdet \begin{pmatrix} \frac{1}{1!} h_{[1]}& \frac{1}{3!} h_{[3]}\\
					\frac{1}{0!} h_{[0]}& \frac{1}{2!} h_{[2]}
			\end{pmatrix}} =  321 \circ \bdet \begin{pmatrix} \frac{1}{1!} h_{1}& \frac{1}{3!} h_{123}\\
				\frac{1}{0!} h_\emptyset& \frac{1}{2!} h_{12}
			\end{pmatrix}\\
			&= 321 \circ \left( \frac{1}{1!} h_{1}\frac{1}{2!} h_{12}  - \frac{1}{3!} h_{123}\frac{1}{0!} h_\emptyset \right) = 321 \circ \left( {\frac{1}{2} h_{1\slashp 12} - \frac{1}{6} h_{123}} \right)\\
			&= 321 \circ \left( {\frac{1}{2} h_{1/ 23} - \frac{1}{6} h_{123}} \right) = \frac{1}{2} h_{12/3} - \frac{1}{6} h_{123}.
		\end{align*}Our classification will explain why \emma{$s_{(213,(2,1))} = s_{(321, (2,2)/(1))} $} later.
	\end{example}
	
	In particular, if $\lambda/\mu$ is a ribbon corresponding to a composition $\alpha$, then the \emph{ribbon Schur function in $\NCSym$}, $r_{[\alpha]}$, is given by \cite[Corollary 6.4]{alw:22}
	\begin{align}\label{E:ribbformula}
		r_{[\alpha]} = (-1)^{\ell(\alpha)} \sum _{\beta \succcurlyeq \alpha} (-1)^{\ell(\beta)}\tcm{\frac{h_{[\beta]}}{\beta!}.}
	\end{align}
	
	\begin{example}\label{EX:noncommralpha} $r_{[(2,1)]}= \frac{1}{2} h_{[(2,1)]}-\frac{1}{6}h_{[(3)]}$
	\end{example}

	Connecting our two Hopf algebras $\Sym$ and $\NCSym$ is the homomorphism 
	$$\rho:\mathbb{Q}\langle\langle x_1,x_2,\ldots\rangle\rangle\rightarrow
	\mathbb{Q}[[x_1,x_2,\ldots]]$$that lets the variables commute, and explicitly relates the functions we have met as follows.
	
	\begin{lemma}\label{L:rho}
		\
		
		\begin{enumerate}
			\item \cite[Theorem 2.1]{rs:04} $\rho(h_\pi) = \lambda(\pi)! h_{\lambda(\pi)}$
			\item \cite[Lemma 4.4]{alw:22} $\rho(s_{(\delta,\lambda/\mu)})=s_{\lambda/\mu}$
			\item \cite[Corollary 6.4]{alw:22} $\rho (r_{[\alpha]}) = r_\alpha$
		\end{enumerate}		
	\end{lemma}

	As a direct consequence, it follows that necessary conditions for the equality of two skew Schur functions are also necessary conditions for the equality of two skew Schur functions in $\NCSym$, and so one ingredient for the proof of our classification is the following necessary condition, and it is in fact the final ingredient necessary to prove our classification.
	
	\begin{lemma}\cite[Corollary 8.11]{rsw:07} \label{L:overlap}
		For two skew diagrams $\D$ and $\T$, if $s_{\D}=s_{\T}$, then $\D$ and $\T$ must have the same $k$-row overlap partitions for all $k$.
	\end{lemma}
	
	\emma{ Another application of the necessary conditions for the equality of two skew Schur functions is 
		the reduction described in Section \ref{S:intro}, which states that the characterization of equality of two skew Schur functions in $\NCSym$ is reduced to determining the equality of two skew Schur functions in $\NCSym$ corresponding to connected skew diagrams. This is true by noting the following necessary condition.
	\begin{lemma}\cite[Section 6]{rsw:07}\cite[Proposition 3.2]{mw:09}
	Understanding the equality $s_{\D}=s_{\T}$ for all skew diagrams $\D$ and $\T$ is equivalent to 
	understanding $s_{\D}=s_{\T}$ for connected skew diagrams $\D$ and $\T$.
	\end{lemma}}
	
	\section{Proof of our main result}\label{S:proof}
	
	We are now ready to prove our main result, in which we classify when two skew Schur functions in $\NCSym$ are equal, 
	recalling that a skew diagram $\D$ has its $180^\circ$ antipodal rotation denoted by $\D ^\ast$, is  nonsymmetric if $\D \neq \D^\ast$, and is a ribbon \tcm{if} each pair of adjacent rows overlap in exactly one column. 	
	\begin{theorem}\label{T:classification} Given two connected skew diagrams $\D$ and $\T$ such that $\D\ne \T$, we have that the skew Schur functions in $\NCSym$ $$s_{(\delta, \D)} =s_{(\tcm{\tau},\T)}$$
		if and only if 
		\begin{enumerate}
			\item $\D$ is a nonsymmetric ribbon and
			\item $\T=\D^*$, and
			\item the bijection $\tdr : j\mapsto n+1-\td(j)$ preserves each block of the set partition $[\alpha]$ where $n=\vert \D\vert$ and $\alpha$ is the  row-length composition of $\D$.
		\end{enumerate}
	\end{theorem}
\emmaj{An equivalent statement of condition (3) is that the bijection $\td$ sends each block $[\alpha_i]_{\alpha}$ of $[\alpha]$ to the block $[\alpha^*_{\ell(\alpha)-i+1}]_{\alpha^*}$ of $[\alpha^*]$ for all $1\le i\le \ell(\alpha)$. See the example below.
	\begin{example}
		Let $\D$ be the ribbon with row-length composition $\alpha=\alpha(\D)=(3,2,1,5,2)$, then $n=|\D|=13$,  $[\alpha]=123/45/6/7\,8\,9\,10\,11/12\,13$ and $\ell(\alpha)=5$. 
		According to condition (2) of Theorem \ref{T:classification}, $\T=\D^*$, $\alpha(\T)=\alpha^*=(2,5,1,2,3)$ and $[\alpha^*]=12/34567/8/9\,10/11\,12\,13$. 
		Let $\tdr=3125469\,11\,10\,7\,8\,13\,12\in S_{13}$ in one-line notation,  then clearly $\tdr$ has property (3), which is true if and only if $\td=11\,13\,12\,9\,10\,85347612$ transforms the sequence of blocks $(123,45,6,7\,8\,9\,10\,11,12\,13)$ to $(11\,12\,13,9\,10,8,34567,12)$.
	\end{example}}
	\begin{proof}    
		In order to make the notation in our proof easier to read, if $\D,\T$ are skew diagrams of size $n$ and $\delta,\tau\in S_n$, then we write $(\delta,\D)\sim(\tau,\T)$, if $s_{(\delta,\D)}=s_{(\tau,\T)}$.  Additionally, we denote $(\id,\T)$ by $(\T)$ where $\id$ is the identity permutation. 
		
		Note that $(\delta,\D)\sim(\tau,\T)$, that is, $\delta\circ s_{[\D]}=\tau\circ s_{[\T]}$ if and only if $(\td)\circ s_{[\D]}=s_{[\T]}$, namely $(\td,\D)\sim(\T)$. Hence it suffices to only prove the sufficient and necessary conditions in this latter case. 
		
		Throughout we also use the notations of our Lemmas in the previous section. 

		For one direction, if $\D$ is a ribbon corresponding to a composition $\alpha$, and $\T=\D^*\ne \D$, then by comparing entries on the main diagonal of $JT(\D)$ and $JT(\T)$ we see that $(\td)\circ h_{[\alpha]}=h_{(\td)[\alpha]}=h_{[\alpha^*]}$
		since the bijection $\tdr : j\mapsto n+1-\td(j)$ preserves each block of $[\alpha]$. For any composition $\beta$ satisfying $\alpha\tcm{\preccurlyeq}\beta$, we have $[\alpha]\leq [\beta]$ and thus the greatest lower bound of all $[\beta]$ such that $\alpha\tcm{\preccurlyeq} \beta$ equals $[\alpha]$. As a result,  $(\td)\circ h_{[\beta]}=h_{[\beta^*]}$ for all $\beta\succcurlyeq\alpha$. In view of \eqref{E:ribbformula}, we find that $(\td)\circ r_{[\alpha]}=r_{[\alpha^*]}$, namely $(\delta, \D) \sim (\tau,\T)$. This completes the proof in this direction.
		
		For the other direction, suppose that $(\delta,\D)\sim(\tau,\T)$. Then we have $(\td)\circ s_{[\D]}=s_{[\T]}$, that is, by \eqref{E:skewsource} with $\ell(\lambda)=\ell$ and $\A_{ij}=\lambda_i-\mu_j-i+j$, 
		\emma{
			\begin{align}\label{E:schurT}
				s_{[\T]}
				=\sum_{w\in S_{\ell}}\mathrm{sgn}(w)\left(\prod_{i=1}^{\ell}(\A_{i\,w(i)}!)^{-1}\right)
				h_{(\gamma^{-1}\delta)[w_{\A}]}.
			\end{align}
			}
		Let $$\alpha=(\A_{11},\ldots,\A_{ii},\A_{i+1\,i+1},\ldots,\A_{\ell\ell})\mathrm{\ and\ }\gamma=(\A_{11}, \ldots,\A_{i\,i+1},\A_{i+1\,i},\ldots , \A_{\ell\ell})$$be the two (weak) compositions $w_{\A}$ corresponding to when $w=\id$ and \emma{ $w=(i,i+1)$}, respectively. So in particular, 
		$\alpha$ is our usual row-length composition \emma{ of $\D$}. 
		
		\emma{ We notice that no cancellation occurs in (\ref{E:schurT}), because $(\gamma^{-1}\delta)[w_{\A}]=(\gamma^{-1}\delta)[\pi_{\A}]$ for $w,\pi\in S_{\ell}$ if and only if $[w_{\A}]=[\pi_{\A}]$ if and only if $w=\pi$. As a consequence, Lemma \ref{L:2} \emma{ and (\ref{E:schurT})} guarantee} that $(\td)[\alpha]=[\mu]$ and $(\td)[\gamma]=[\nu]$ for some $\mu,\nu\vDash n$. 
		Both $[\mu]$ and $[\nu]$ have common blocks $(\td)[\A_{jj}]_{\alpha}$ for all $j\not\in\{i,i+1\}$ since $[\alpha]$ and $[\gamma]$ have common blocks $[\A_{jj}]_{\alpha}=[\A_{jj}]_{\gamma}$ for all $j\not\in\{i,i+1\}$. 
		
		We claim that two blocks $(\td)[\A_{ii}]_{\alpha}$ and $(\td)[\A_{i+1\,i+1}]_{\alpha}$ must be adjacent in the sense that the union of $(\td)[\A_{ii}]_{\alpha}$ and $(\td)[\A_{i+1\,i+1}]_{\alpha}$ is a set of consecutive integers.
		
		To see this, since $$\A_{ii}+\A_{i+1\,i+1}=\A_{i\,i+1}+\A_{i+1\,i},$$by Lemma \ref{L:1}, and $\A_{ii}<\A_{i\,i+1}$, by the definition of the Jacob-Trudi matrix, we have that 
		\begin{align}\label{E:asubsets}[\A_{ii}]_{\alpha}\subsetneq[\A_{i\,i+1}]_{\gamma}\subseteq
			[\A_{i\,i+1}]_{\gamma}\cup   [\A_{i+1\,i}]_{\gamma}
			=[\A_{ii}]_{\alpha}\cup   [\A_{i+1\,i+1}]_{\alpha}.\end{align}
		It follows that 
		\begin{align}\label{E:a1}
			(\td)[\A_{ii}]_{\alpha}\subsetneq
			(\td)[\A_{i\,i+1}]_{\gamma}\subseteq (\td)[\A_{ii}]_{\alpha}\,\cup (\td)[\A_{i+1\,i+1}]_{\alpha}. 
		\end{align}
		Note that $(\td)[\A_{i\,i+1}]_{\gamma}$ is a block of $[\nu]$, which by definition is a set of consecutive integers, implying that $(\td)[\A_{ii}]_{\alpha}$ and $(\td)[\A_{i+1\,i+1}]_{\alpha}$ must be adjacent as claimed. 
		Consequently, because this is true for all $1\leq i \leq \ell-1$, we have that either
		
		\begin{enumerate}[label=(\Roman*)]
			\item the permutation $\td$ preserves every block of $[\alpha]$, \emmaj{ that is, the subscripts on the main diagonal of $JT(\T)$ are $\A_{11},\ldots,\A_{\ell\ell}$ from northwest to southeast}, or
			\item the bijection $\tdr : j\mapsto n+1-\td(j)$ preserves every block of $[\alpha]$, \emmaj{namely, the subscripts on the main diagonal of $JT(\T)$ are $\A_{\ell\ell},\ldots,\A_{11}$ from northwest to southeast.}
		\end{enumerate}
		
		Before we analyze these two cases, we observe that $(\td)[\A_{1\ell}]=[\A_{1\ell}]$. This is true because $[\A_{1\ell}]$ is the unique largest block of any set partition $[w_{\A}]$ and $(\td)[w_{\A}]$ as a result of (1) of Lemma \ref{L:JTsub}. Now we begin our analysis.
		
		We begin with the simpler case \emmaj{(II)}, where $\tdr : j\mapsto n+1-\td(j)$ preserves every block of $[\alpha]$. 
		In this case we have that 
		$$\A_{11}+\cdots +\A_{\ell\ell}=\A_{1\ell}$$because of the following. First note that since  $\tdr : j\mapsto n+1-\td(j)$ preserves every block of $[\alpha]$, this implies that $1\in \tdr [\A_{11}]$, and hence that $n\in (\td)[ \A_{11}]$. From our observation above we also know that $(\td)[\A_{1\ell}]=[\A_{1\ell}]$. Therefore, since $\A_{11} < \A_{1\ell}$ (for $\ell \neq 1$) we obtain that
		$$n\in (\td)[ \A_{11}]\subsetneq (\td)[\A_{1\ell}]=[\A_{1\ell}]$$implying that $\A_{1\ell}\geq n$. Second,  note that
		$\A_{1\ell}\leq n$ because $\A_{1\ell}$ is a part of a composition of $n$. 
		
		Consequently, since $\A_{1\ell}\geq n$ and $\A_{1\ell}\leq n$, we have that $$\A_{1\ell}=n=\A_{11}+\cdots +\A_{\ell\ell}.$$ By repeatedly applying Lemma \ref{L:1}, we are led to $$\A_{1\ell}=\A_{11}+\cdots +\A_{\ell\ell}=\A_{1\ell}+\A_{21}+\emma{\A_{32}}+\cdots +\A_{\ell\,\ell-1},$$ which gives that $\A_{j+1\,j}=0$ for all $1\leq j \leq \ell -1$,  because $\D$ is connected and so by definition $\A_{j+1\,j}\geq 0$. Hence, $s_{[\D]}$ is the determinant of a matrix whose subdiagonal entries are all \emma{$h_{[0]}=1$}, which implies that the row overlap of each pair of adjacent rows in $\D$ is 1, that is, $\D$ is a ribbon. Now if  $(\td,\D)\sim(\T)$, then $s_{(\td,\D)}=s_{(\id, \T)}$ and so
		\begin{align}\label{E:rho}s_\D = \rho(s_{(\td,\D)})= \rho(s_{(\id, \T)}) = s_\T\end{align}according to Lemma~\ref{L:rho}. \emmaj{Therefore by Lemma \ref{L:overlap}, $\T$ is a ribbon. Under the assumption of (II), the subscripts on the main diagonal of $JT(\T)$ are $\A_{\ell\ell},\ldots,\A_{11}$ from northwest to southeast, thus we find that $\T=\D^*$ and $\D^*\ne \D$ by assumption that $\D\ne \T$. This completes the case  (II)}.

		For the other case \emmaj{(I)}, where $\td$ preserves every block of $[\alpha]$, we will show that the only possible way to realize $(\td, \D)\sim (\T)$ is that $\D=\T$. 
		
		Since $(\td) [\alpha]=[\alpha]$ by assumption, and $(\td)\circ h_\pi = h_{(\td)\pi}$ for any set partition $\pi$, and we know \eqref{E:a1}, we obtain that 
		\begin{align}\label{E:subcase2}
			[\A_{ii}]_{\alpha}\subsetneq (\td)[\A_{i\,i+1}]_{\gamma}\subseteq [\A_{ii}]_{\alpha}\cup[\A_{i+1\,i+1}]_{\alpha}= [\A_{i\,i+1}]_{\gamma}\cup[\A_{i+1\,i}]_{\gamma},
		\end{align}
		by \eqref{E:asubsets}.
		Consequently, 
		because we know that  $(\td)[\A_{i\,i+1}]_\gamma$ and $(\td)[\A_{i+1\,i}]_{\gamma}$ are blocks of $[\nu]$ they must each be a set of consecutive integers,  \emma{ implying $[\A_{ii}]_{\alpha}\subsetneq (\td)[\A_{i\,i+1}]_{\gamma}\subseteq [\A_{i\,i+1}]_{\gamma}$,}
		and so $(\td)[\A_{i\,i+1}]_{\gamma}=[\A_{i\,i+1}]_{\gamma}$. 
		\emma{ This further leads to}  $(\td)[\A_{i+1\,i}]_{\gamma}=[\A_{i+1\,i}]_{\gamma}$ in view of 
		$(\td)([\A_{i\,i+1}]_{\gamma} \cup [\A_{i+1\,i}]_{\gamma})=[\A_{i\,i+1}]_{\gamma} \cup [\A_{i+1\,i}]_{\gamma}$,
		which follows from the last equality of (\ref{E:subcase2}) and the assumption that $(\td) [\alpha]=[\alpha]$.

		We also have that $$(\td)[\A_{jj}]_{\gamma}=(\td)[\A_{jj}]_{\alpha}= [\A_{jj}]_{\alpha} = [\A_{jj}]_{\gamma}$$for all $j\not\in\{i,i+1\}$ since $[\alpha]$ and $[\gamma]$ have common blocks $[\A_{jj}]_{\alpha} = [\A_{jj}]_{\gamma}$ 	for all $j\not\in\{i,i+1\}$, and so we have established that $\td$ preserves every block of $[\gamma]$, that is, $(\td)[\gamma]=[\gamma]$.
		
		We will now keep \eqref{E:rho} in mind to employ both the commutative and noncommutative settings at once, in order to obtain our result  for noncommuting variables. This we will do by determining the subscripts on the main diagonal and subdiagonal entries of the Jacobi-Trudi matrix $JT(\T)$.
		
		First we determine the subscripts on the main diagonal entries of the Jacobi-Trudi matrix $JT(\T)$. In view of Lemma \ref{L:overlap}, the $1$- and $2$-row overlap partitions of $\D$ and $\T$ are the same, that is, the \emph{set} of subscripts on the main diagonal and subdiagonal entries of $JT(\D)$ are the same as the ones of $JT(\T)$, respectively. By $(2)$ $(a)$--$(b)$ of Lemma \ref{L:JTsub}, only the identity permutation, $\id$, gives rise to the term with subscripts $\A_{jj}$ for $1\leq j \leq \ell$ in \emmaj{$s_{\D}$ and $s_{\T}$} respectively. As a result, it follows from $(\td) [\alpha]=[\alpha]$ that the subscript of the \tcm{$(j,j)$th} entry of $JT(\T)$ must be $\A_{jj}$. Thus, we conclude that the subscripts on the main diagonal entries of the Jacobi-Trudi matrix $JT(\T)$ are identical to those  of the Jacobi-Trudi matrix $JT(\D)$. 
		
		Now we turn our attention to the subscripts on the subdiagonal entries of the Jacobi-Trudi matrix $JT(\T)$.
		Because we verified earlier that $(\td) [\gamma]=[\gamma]$, the (weak) composition $\gamma = (\A_{11},\ldots,\A_{i\,i+1},\A_{i+1\,i},\ldots,\A_{\ell\ell})$ appears as the subscript of a complete homogeneous symmetric function in the determinantal expansion of $JT(\T)$. Since the $\A_{jj}$  are the subscripts of the \tcm{$(j,j)$th} entries on the main diagonal of $JT(\T)$, by above, and 
		$$\A_{i\,i+1}>\A_{ii}>\A_{i+1\,i}.$$
		the subscript of the \tcm{$(i+1,i)$th} entry of $JT(\T)$ must be $\A_{i+1\,i}$. 
		
		Because this is true for all $1\leq i \leq \ell-1$, we obtain that the subscripts on the subdiagonal entries of the Jacobi-Trudi matrix $JT(\T)$ are identical to those  of the Jacobi-Trudi matrix $JT(\D)$.  Thus, $$JT(\T)=JT(\D)$$because the subscripts on the main diagonal and subdiagonal entries of any Jacobi-Trudi matrix determine the remaining subscripts of the remaining entries by the definition of the Jacobi-Trudi matrix (\emmaj{or equivalently, the conclusion of Lemma \ref{L:1}}). This leads to $\D=\T$. 
		This completes the case \emmaj{(I)}. 
		
		Consequently, putting both these cases together, if $\D\ne \T$, then $\D$ must be a ribbon and $\T=\D^*\ne \D$. This completes the proof of the other direction, and we are done.\end{proof}
	
	\section*{Acknowledgements}\label{sec:ack}
	Both authors would like to thank the referee for their valuable suggestions and careful reading.
	The first author was supported by the Austrian Research Fund FWF Elise-Richter Project V 898-N, is supported by the Fundamental Research Funds for the Central Universities, Project No. 20720220039 and the National Nature Science Foundation of China (NSFC), Project No. 12201529. The second author is supported in part by the Natural Sciences Engineering and Research Council of Canada. 
	

\end{document}